\documentclass[11pt, reqno]{amsart}

\usepackage{color}
\usepackage{amsmath}
\usepackage{latexsym,amssymb}
\usepackage{graphicx}
\usepackage{stmaryrd}
\usepackage{bm}
\usepackage{enumerate}
\usepackage{setspace}
\newtheorem{thm}{ \bf Theorem}[section]

\newtheorem{lem}[thm]{ \bf Lemma}
\newtheorem{prop}[thm]{ \bf Proposition}

\newtheorem{rem}[thm]{ \bf Remark}

\newtheorem{assumption}[thm]{ \bf Assumption}
\numberwithin{equation}{section}

\newcommand{\noi}{\noindent}

\newcommand{\R}{\mathbb{R}}
\newcommand{\N}{\mathbb{N}}

\newcommand{\sig}{\sigma}

\newcommand{\eps}{\varepsilon}
\newcommand{\io}{\iota}
\newcommand{\ph}{\varphi}

\newcommand{\al}{\alpha}

\newcommand{\del}{\delta}

\newcommand{\Gam}{\mathnormal{\Gamma}}
\newcommand{\Del}{\mathnormal{\Delta}}

\newcommand{\La}{\mathnormal{\Lambda}}

\newcommand{\Ups}{\mathnormal{\Upsilon}}
\newcommand{\Om}{\mathnormal{\Omega}}

\newcommand{\EE}{{\mathbb E}}

\newcommand{\PP}{{\mathbb P}}

\newcommand{\calI}{\mathcal{I}}

\newcommand{\calQ}{\mathcal{Q}}

\newcommand{\AC}{\mathcal{AC}}

\DeclareMathOperator*{\argmax}{arg\,max}

\newcommand{\oo}{\overline}
\newcommand{\skp}{\vspace{\baselineskip}}

\newcommand{\iy}{\infty}
\newcommand{\osc}{\text{osc}}
\newcommand{\ds}{\displaystyle}

\newcommand{\IA}{\text{\it IA}}
\newcommand{\ST}{\text{\it ST}}

\begin{document}

\baselineskip=17pt

\title[]
{Optimality of the Generalized $\boldsymbol{c\mu}$ Rule\\ in the Moderate Deviation Regime}

\author[Rami Atar]{Rami Atar}
\address{Department of Electrical Engineering\\
Technion -- Israel Institute of Technology\\
Haifa, Israel}
\email{atar@ee.technion.ac.il}

\author[Subhamay Saha]{Subhamay Saha}
\address{Department of Mathematics\\
Indian Institute of Technology Guwahati\\
Guwahati, Assam, India}
\email{saha.subhamay@iitg.ernet.in}

\thanks{Research supported in part by the ISF (Grant 1315/12)}


\date{}

\begin{abstract}
This paper studies a multiclass queueing system with an associated
risk-sensitive cost observed in heavy traffic at the moderate deviation scale,
accounting for convex queue length penalties.
The main result is the asymptotic optimality of a dynamic index policy
known from the diffusion scale heavy traffic literature as the generalized $c\mu$ rule.

\skp

\noi
{\bf AMS subject classifications:}
60F10, 60K25, 93E20, 49N70

\skp

\noi
{\bf Keywords:}
moderate deviations; heavy traffic; risk-sensitive cost; differential games
\end{abstract}

\maketitle

\section{{\bf Introduction}}\label{sec1}

One of the most appealing results on queueing control problems (QCPs) in heavy traffic,
obtained by van Mieghem \cite{van},
is the asymptotic optimality (AO) of a dynamic index policy for the
multiclass queue with nonlinear holding (as well as queue length) costs.
This policy, referred to as the {\it generalized $c\mu$ rule},
acts as a localized version of the classical $c\mu$ rule, where the parameter $c$ is variable
and obtained by feedback from the system's state.
Specifically, a class-$i$ customer that experiences a delay of $\del$ units of time incurs a cost
$C_i(\del)$, where $C_i$ are given smooth, convex functions.
The rule is to prioritize the classes according to the index
$\mu_iC'_i(\hat Q_i^n(t))$, where $n$ is the scaling parameter,
$\mu_i$ and $\hat Q_i^n(t)$ are the corresponding service rate and
{\it diffusion scaled} queue length at time $t$, and $C'_i$ denotes the derivative of $C_i$.
The index is thus explicit and simple to compute, and in particular does not require
solving a dynamic programming equation, where at the same time, it is valid for a relatively
rich family of costs.

This work addresses the multiclass scheduling problem in the
{\it moderate deviation heavy traffic} (MDHT) regime,
with a {\it risk-sensitive} (RS) version of the above cost.
More precisely, assuming renewal structure for
the arrival and potential service time processes, the cost considered is
\[
\frac{1}{b_n^2}\log\EE\exp\Big\{b_n^2\int_0^T\sum_iC_i(\tilde Q^n_i(t))dt\Big\},
\]
for a sequence $b_n\to\iy$ with $n^{-1/2}b_n\to0$, where
$\tilde Q^n_i:=b_n^{-1}n^{-1/2}Q^n_i$ denotes the MD scaled queue length for class $i$,
and arrival and service rates both scale like $n$.
The main result is the MD scale AO
of the generalized $c\mu$ rule, with the obvious adaptation of determining the index
via $\tilde Q^n_i$ rather than $\hat Q^n_i$.

In that, this paper contributes to a line of research started in \cite{AtaBis} and continued in
\cite{ata-coh} and \cite{ata-coh-2}, that addresses QCPs at the MDHT regime.
The reader is referred to these papers for background and motivation,
as well as discussions of
similarities and dissimilarities between diffusion scale heavy traffic approximations and MDHT,
as well as between the latter and QCPs at the large deviation (LD) regime.
It has been suggested in \cite{AtaBis} that RS control at the MDHT regime is likely
to enjoy explicitly computable AO policies much more often than at the LD regime.
The papers \cite{ata-coh} and \cite{ata-coh-2} support this expectation by establishing AO
of explicitly computable policies for a complex QCP that involves a free boundary problem.
Put in this context, the contribution of the present paper further strengthens the
assertion from \cite{AtaBis} alluded to above.
In relation to this, it should be mentioned that an AO result of the generalized $c\mu$ rule
is not to be expected in the LD regime. In fact, when
$C_i$ are linear functions, it is known under some conditions
(including Markovity) that a RS control problem in the LD regime
possesses an AO index rule distinct from the classical $c\mu$ rule
(see \cite{AGS2} where a different index rule is shown to be AO in this regime).

The aforementioned paper \cite{AtaBis} studies a general RS cost structure, and attains
AO of policies described in terms of an underlying differential game.
However, the main focus in \cite{AtaBis} is to deal with
the {\it existence} of such polices,
not with the construction of ones that are in any reasonable sense described explicitly.
Indeed, the sole example of a cost function $C$ for which
an explicit AO policy is known for this setting is that where $C$ is linear \cite{AtaBis}.
It is desired to extend the family of cost functions possessing explicit AO policies,
and the contribution of this paper can indeed be viewed as such an extension
to the collection of convex functions $C$ that adhere to a linear growth condition
and certain smoothness assumptions.

Our proof technique borrows from \cite{AtaBis} results regarding this game,
as well as a general lower bound on the RS performance of the control system.
The main body of work needed to prove our AO result consists of establishing
a matching upper bound on the performance under the policy considered.

Let us finally mention that the diffusion scale heavy traffic analysis of
the generalized $c\mu$ rule has reached far beyond \cite{van}, most notably in
\cite{ManStol}, where the rule's AO was proved in a multi-server setting that allows for servers
that differ from each other in terms of the service capabilities offered to the various classes.
It is of interest to study this type of extension at the MD scale in future work.

The setting and main results appear in Section \ref{sec2}.
Specifically, Theorems \ref{mainthm1} and \ref{mainthm2} assert the AO of a preemptive
and nonpreemptive version of the rule, respectively.
Section~\ref{sec3} provides some elementary facts about the differential game,
while Sections \ref{sec41} and \ref{sec42} give the proofs of
Theorems \ref{mainthm1} and \ref{mainthm2}, respectively.

\skp

{\it Notation.}
For a positive integer $k$ and $\alpha,\beta \in \mathbb{R}^k$, $\alpha\cdot \beta$ will denote the
usual inner product, while $\|\cdot\|$ will denote the Euclidean norm. For $T>0$ and a
function $f:[0,T]\rightarrow \mathbb{R}^k$, let $\|f\|_t=\sup_{0\leq s\leq t}\|f(s)\|, t \in [0,T]$.
For $\delta >0$, the $\delta$-oscillation of $f$ is defined as
$$
\mbox{osc}_{\delta}(f)=\sup\{\|f(s)-f(t)\|:|s-t|\leq \delta,\, s,t \in [0,T]\}\,.
$$
For a function $f$ defined on $\R_+$ and $y>0$, $\text{osc}_\del(f,y)$ denotes
the $\del$-oscillation of $f|_{[0,y]}$.
Denote by $\mathcal{C}([0,T],\mathbb{R}^k)$ the space of continuous functions. The space of
functions that are right-continuous with finite left limits (RCLL) will be denoted by
$\mathcal{D}([0,T],\mathbb{R}^k)$,
and endowed with the Skorohod $J_1$ metric \cite{Bill}. It is a well known fact (see \cite{Bill})
that this metric space is Polish. Let $\AC([0,T],\R^k)$
denote the class of absolutely continuous mappings from $[0,T]$ to
$\R^k$. Let $\AC_0([0,T],\R^k)$
denote the set of $\psi\in\AC([0,T],\R^k)$ with $\psi(0)=0$.

\section{{\bf Model Description and Results}}\label{sec2}

The model consists of $d$ customer classes and a single server. An infinite capacity buffer is devoted to each customer class, where customers are queued if the server is not available to serve them upon arrival. Within each class the customers are served on first come first served basis. The server may serve only one customer at a time within a class, however, processor sharing is allowed, in the sense that the server can divide its effort among up to $d$ customers of distinct classes.

The model is defined on a complete probability space $(\Om, \mathcal{F}, \mathbb{P})$. Expectation with respect to $\mathbb{P}$ is denoted by $\mathbb{E}$. The parameters and the processes will be indexed by $n \in \N$, which will act as the scaling parameter. The arrivals occur according to independent renewal processes. Let $\mathcal{I}=\{1,\dots,d\}$. Let $\lambda^n_i, n\in\N, i\in \mathcal{I}$ be given parameters, representing the reciprocal of the mean arrival times for class-$i$ customers. Given are $d$ independent sequences $\{\IA_i(l):l\in \N\}_{i \in \mathcal{I}}$, of i.i.d. positive random variables with mean $1$ and variance $\sigma^2_{\IA,i}$. Then the number of arrivals of class-$i$ customers up to time $t$ is given by
$$A^n_i(t)=\sup \biggl\{l\geq 0:\sum_{k=1}^l\frac{\IA_i(k)}{\lambda_i^n}\leq t\biggr\},\quad t\geq 0\,.$$
Similarly, we consider another set of parameters $\mu^n_i, n\in\N, i\in \mathcal{I}$, representing the reciprocal mean service times. Also given are $d$ independent sequences $\{\ST_i(l):l\in \N\}_{i \in \mathcal{I}}$ of positive i.i.d. random variables with mean $1$ and variance $\sigma^2_{\ST,i}$. The potential service process for class-$i$ customers is given by
$$S^n_i(t)=\sup \biggl\{l\geq 0:\sum_{k=1}^l\frac{\ST_i(k)}{\mu_i^n}\leq t\biggr\},\quad t\geq 0\,.$$
Namely, $S^n_i(t)$ gives the number of class-$i$ customers served by the time the server has devoted
$t$ units of time to that class.
Now let $\{b_n\}$ be a sequence satisfying $\lim b_n = \infty$ and $\lim \frac{b_n}{\sqrt{n}}=0$.
We will refer to $b_n^{-2}$ as the {\it moderate deviation speed}.
The arrival and service parameters are assumed to satisfy the following limits as $n \rightarrow \infty$:
\begin{equation}\label{56}
\begin{split}
&\text{$\lambda^n_i/n \rightarrow \lambda_i \in (0, \infty)$ and $\mu^n_i/n \rightarrow \mu_i \in (0, \infty)$,}\\
&\text{$\tilde{\lambda}^n_i:= \frac{1}{b_n\sqrt{n}}(\lambda_i^n-n\lambda_i)\rightarrow \tilde{\lambda}_i \in (-\infty, \infty)$,}\\
&\text{$\tilde{\mu}^n_i:= \frac{1}{b_n\sqrt{n}}(\mu_i^n-n\mu_i)\rightarrow \tilde{\mu}_i \in (-\infty, \infty)$.}
\end{split}
\end{equation}
Also, the system is assumed to be critically loaded, that is, $\sum_{i=1}^d \rho_i=1$, where $\rho_i=\lambda_i/\mu_i, i \in \mathcal{I}$. For $i \in \mathcal{I}$, let $Q^n_i$ denote the number of
class-$i$ customers in the system. With $S= \{x=(x_1,\dots,x_d)\in [0,1]^d:\sum x_i \leq 1\}$, let $B^n$ be a process taking values in $S$, with the $i$th component representing the fraction of effort given to the class-$i$ customer. Then the number of class-$i$ jobs completed by time $t$ is given by
\begin{align}\label{50}
D^n_i(t):= S^n_i(T^n_i(t)),
\end{align}
where
\begin{align}
T^n_i(t)=\int_0^t B^n_i(s)ds
\end{align} is the cumulative amount of time devoted to class-$i$ customers by time $t$.
Assuming, for simplicity, that the system starts empty, we have
\begin{align}\label{ioeq}
Q^n_i(t)= A^n_i(t)-D^n_i(t)\,.
\end{align}
We regard $B^n$ as the {\it control process}. Given $n$, we say that $B^n$ is an
{\it admissible} control if it has RCLL sample paths, and
\begin{itemize}
\item it is adapted to the filtration
$$\sigma\{A^n_i(s), S^n_i(T^n_i(s)), i \in \mathcal{I}, s\leq t\},$$
\item for every $i \in \mathcal{I}$ and $t\geq 0$, one has,
\begin{align}\label{52}
Q^n_i(t)=0 \text{ implies } B^n_i(t)=0\,.
\end{align}
\end{itemize}
Denote the set of admissible controls by $\mathcal{B}^n$. An admissible control $B^n$
is said to be {\it non-preemptive} if (i) processor sharing is not allowed,
namely $B^n$ takes values in $\{x\in\{0,1\}^d:\sum x_i\le 1\}$; and (ii)
service is non-interruptible, namely
if for some $t$ and $i$ one has $B^n_i(t)=1$, then $B^n_i(u)=1$ for $u\in[t,\tau)$,
where $\tau$ is the time of next departure,
$\tau:=\inf\{s\ge t:\Del^n(s)>\Del^n(t)\}$, $\Del^n:=\sum_jD^n_j$.
Denote the set of nonpreemptive admissible controls by $\mathcal{B}^{\#,n}$.
Clearly, $\mathcal{B}^{\#,n}\subset\mathcal{B}^n$.

We now introduce the centered and scaled versions of the processes,
\begin{align}\label{51}
\begin{split}
&\tilde{A}^n_i(t)= \frac{1}{b_n\sqrt{n}}(A^n_i(t)-\lambda^n_i t), \quad \tilde{S}^n_i(t)= \frac{1}{b_n\sqrt{n}}(S^n_i(t)-\mu^n_i t),\\
&\tilde{Q}^n_i(t)= \frac{1}{b_n\sqrt{n}}Q^n_i(t)\,.
\end{split}
\end{align}
Then it follows from \eqref{ioeq} that,
\begin{align}\label{systeq}
\tilde{Q}^n_i(t)=y^n_i t + \tilde{A}^n_i(t) - \tilde{S}^n_i(T^n_i(t))+ Z^n_i(t),
\end{align}
where
\begin{align}
Z^n_i(t)= \frac{\mu^n_i}{n}\frac{\sqrt{n}}{b_n}(\rho_i t - T^n_i(t)), \quad y^n_i = \tilde{\lambda}^n_i - \rho_i \tilde{\mu}^n_i\,.
\end{align}
The scaled processes $(\tilde{A}^n, \tilde{S}^n)$ are assumed to satisfy a moderate deviation principle. For that let $I_k, k=1,2$, be functions defined on $\mathcal{D}([0,T],\mathbb{R}^d)$ as follows.
Denote $\AC_0=\AC_0([0,T],\R^d)$.
For $\psi = (\psi_1,\dots, \psi_d) \in \mathcal{D}([0,T],\mathbb{R}^d)$,
\[
\mathbb{I}_1(\psi)=\begin{cases}
\ds
\sum_{i=1}^d\frac{1}{2\hat\sig_{1,i}^2}\int_0^T\dot{\psi}^2_i ds,
&\text{if } \psi\in\AC_0,\\
\infty, &\mbox{otherwise},
\end{cases}
\qquad
\mathbb{I}_2(\psi)=\begin{cases}
\ds
\sum_{i=1}^d\frac{1}{2\hat\sig_{2,i}^2}\int_0^T\dot{\psi}^2_i ds,
&\mbox{if } \psi\in\AC_0,\\
\infty, &\mbox{otherwise},
\end{cases}
\]
where
$\hat\sig_{1,i}^2=\lambda_i\sigma^2_{\IA,i}$, $\hat\sig_{2,i}^2=\mu_i\sigma^2_{\ST,i}$.
Let $\mathbb{I}(\psi)= \mathbb{I}_1(\psi^{(1)})+ \mathbb{I}_2(\psi^{(2)})$ for $\psi=(\psi^{(1)},\psi^{(2)}) \in \mathcal{D}([0,T],\mathbb{R}^{2d})$. Note that $\mathbb{I}$ is lower semi-continuous and has compact level sets, hence is a good rate function.

\begin{assumption}\label{A1} The sequence $(\tilde{A}^n, \tilde{S}^n)$ satisfies the LDP with
speed $b_n^{-2}$ and rate function $\mathbb{I}$ in $\mathcal{D}([0,T],\mathbb{R}^{2d})$, that is:\\
$\bullet$ for any open set $G \subset \mathcal{D}([0,T],\mathbb{R}^{2d})$
$$\liminf \frac{1}{b_n^2}\log\mathbb{P}((\tilde{A}^n, \tilde{S}^n)\in G)\geq -\inf_{\psi \in G}\mathbb{I}(\psi);$$
$\bullet$ for any closed set $F \subset \mathcal{D}([0,T],\mathbb{R}^{2d})$
$$\limsup \frac{1}{b_n^2}\log\mathbb{P}((\tilde{A}^n, \tilde{S}^n)\in F)\leq -\inf_{\psi \in F}\mathbb{I}(\psi).$$
\end{assumption}
\noi
A sufficient condition for this assumption is the existence of finite exponential
moments for the random variables $\IA_i(1)$ and $\ST_i(1)$ (a precise statement
is provided below, in Assumption \ref{A3}).
For a proof, as well as considerably weaker sufficient conditions, see \cite{PW}.

The cost to be considered is defined in terms of functions $C_i, i \in \mathcal{I}$ that are strictly increasing, strictly convex, non-negative, continuously differentiable functions from $\mathbb{R}_+$ to $\mathbb{R}_+$, satisfying $C_i(0)=C_i^{\prime}(0)=0$. A linear growth condition is assumed,
namely $C_i(x)\le u_1x$ for all $x\in\R_+$, $i\in\calI$, where $u_1$ is a constant.
Denote $C(x)=\sum_iC_i(x_i)$, where $x=(x_1,\ldots,x_d)$.
For more about the assumptions on $C$, see Remark \ref{rem1} below.
Given $n$, the cost, that is of RS type,
associated with the control $B^n \in \mathcal{B}^n$ is given by
\begin{align}
J^n(B^n)=\frac{1}{b_n^2}\log\mathbb{E}\biggl[\exp\biggl(b_n^2\displaystyle\int_0^T
C(\tilde{Q}^n(t))dt\biggr)\biggr]\,.
\end{align}
The two value functions of interest are given by
\begin{align*}
V^n = \inf_{B^n \in \mathcal{B}^n}J^n(B^n)\,,
\qquad
\hat{V}^n = \inf_{B^n \in \mathcal{B}^{\#,n}}J^n(B^n)\,.
\end{align*}

Now we describe the differential game which will govern the asymptotic behavior of the above control problem. Let $\theta = (\mu_1^{-1},\ldots, \mu_d^{-1})$ and let $y=(y_1,\ldots,y_d)$ where $y_i=\tilde{\lambda}_i-\rho_i\tilde{\mu}_i$. Denote $\mathcal{P}=\mathcal{C}_0([0,T],\mathbb{R}^{2d})$, the set of continuous functions starting from $0$, and
$$\mathcal{E}=\{\zeta \in \mathcal{C}([0,T],\mathbb{R}^d):\theta\cdot\zeta\,\,\, \mbox{starts from zero and is non-decreasing}\}\,.$$
Endow both the spaces with the uniform topology. Let $\rho$ be the mapping from $\mathcal{D}([0,T],\mathbb{R}^d)$ into itself defined by
$$\rho[\psi]_i(t)=\psi_i(\rho_i t),\quad t \in [0,T],\,\,i\in \mathcal{I}\,.$$
Given $\psi=(\psi^{(1)},\psi^{(2)}) \in \mathcal{P}$ and $\zeta \in \mathcal{E}$, the {\it dynamics for
the data $(\psi, \zeta)$} is defined as
\begin{align}\label{dyn}
\varphi=y \io + \psi^{(1)}-\rho[\psi^{(2)}]+\zeta,
\end{align}
where $\io:[0,T]\to[0,T]$ is the identity map.
The dynamics are considered with the constraint
\begin{align}\label{nonneg}
\varphi_i(t)\geq 0, \quad t\geq 0, i\in\mathcal{I}\,.
\end{align}
The game is defined in the sense of Elliot and Kalton \cite{EllKal}, for which we need the notion of strategies. A measurable mapping $\alpha:\mathcal{P}\rightarrow \mathcal{E}$ is called a {\it strategy} if it satisfies a causality property. Namely, for every $\psi=(\psi^{(1)},\psi^{(2)}), \tilde{\psi}=(\tilde{\psi}^{(1)},\tilde{\psi}^{(2)}) \in \mathcal{P}$ and $t \in [0,T]$,
$$(\psi^{(1)},\rho[\psi^{(2)}])(s)=(\tilde{\psi}^{(1)},\rho[\tilde{\psi}^{(2)}])(s)\quad \forall s \in [0,t] \quad \mbox{implies}\,\,\quad \alpha[\psi](s)=\alpha[\tilde{\psi}](s)\,\,\quad\forall s \in [0,t]\,.$$
A strategy $\alpha$ is said to be {\it admissible} if, whenever $\psi \in
\mathcal{P}$ and $\zeta=\alpha[\psi]$, the corresponding dynamics \eqref{dyn} satisfies the non-negativity constraint \eqref{nonneg}. We denote the set of all admissible strategies by $\mathcal{A}$. Then
the cost associated with $(\psi, \zeta) \in \mathcal{P}\times \mathcal{E}$ is given by
$$
c(\psi,\zeta)=\int_0^T C(\ph(t))dt-\mathbb{I}(\psi),
$$
where $\ph$ is the dynamics for the data $(\psi,\zeta)$. The value of the game is defined as
$$V=\inf_{\alpha \in \mathcal{A}}\sup_{\psi \in \mathcal{P}}c(\psi,\alpha[\psi])\,.$$
Before stating the main result we require one more assumption.
\begin{assumption}\label{A2}
For any constant $u$,
$$\limsup_{n\rightarrow \infty}\frac{1}{b_n^2}\log\mathbb{E}\biggl[e^{b_n^2u(\|\tilde{A}^n\|_T+\|\tilde{S}^n\|_T)}\biggr]<\infty\,.$$
\end{assumption}
\noi
Again, this assumption holds when $\IA_i(1)$ and $\ST_i(1)$ have finite exponential moments.
More precisely, consider
\begin{assumption}\label{A3}
There exists a $u_0 > 0$ such that $\mathbb{E}[e^{u_0\IA_i(1)}]$ and $\mathbb{E}[e^{u_0\ST_i(1)}]$, $i \in \mathcal{I}$ are finite.
\end{assumption}
\noi
Then Assumption \ref{A3} is a sufficient condition for
Assumption \ref{A1} (by Theorem 6.2 of \cite{PW})
and for Assumption \ref{A2} (by Proposition 2.1 of \cite{AtaBis}).

Now we define a particular control that will be referred to as the {\it preemptive
generalized $c\mu$ rule}. This policy gives preemptive priority to
the class $i$ for which
$\mu_iC_i^{\prime}(\tilde{Q}^n_i)\geq \mu_jC_i^{\prime}(\tilde{Q}^n_j)$
for all $j$, where ties are broken in some predefined manner.
To define it precisely we need some additional notation.
Given a set of $d$ real numbers $A=\{\al_i,i\in\calI\}$,
denote $\argmax A=\{i:\al_i\ge\max_j\al_j\}$, and
let $\argmax^*A$ be the smallest member of
$\argmax A$.
The control, that we denote by $B^{*,n}$, is defined by setting
\begin{align}\label{control}
B^{*,n}_i(t)=1_{\{\tilde{Q}^n(t)\in\calQ_i\}}\,, \qquad i\in\calI,
\end{align}
where $\calQ_i$, $i\in\calI$ partition $\R_+^d\setminus\{0\}$ according to
\begin{equation}\label{58}
\calQ_i=\{q\in\R_+^d\setminus\{0\}:\textstyle\argmax^*\{\mu_jC'_j(q),j\in\calI\}=i\},\qquad i\in\calI.
\end{equation}
(Thus, in case of a tie, priority is given to the lowest index.)
Note that if, for some $i$, $q\in\calQ_i$, we have $q_i>0$ thanks to the assumption that,
for all $i$, $C'_i(x)=0$ iff $x=0$; as a result, the queue selected for service is nonempty.

It is easy to see that equation \eqref{control}, along with equations \eqref{50}--\eqref{ioeq}
and \eqref{51} uniquely define the processes $B^n=B^{*,n}$, $D^n$, $T^n$, $Q^n$
and $\tilde Q^n$, based on the data $(A^n,S^n)$. Moreover,
the process $B^{*,n}$ thus defined is an admissible control.

It follows from the results of \cite{AtaBis} that, under Assumptions \ref{A1} and \ref{A2}, \begin{equation}
\label{10}
\lim_{n\rightarrow \infty}V^n=V.
\end{equation}
While this result is crucial in validating the heuristic that the asymptotics are
governed by the differential game, it leaves open the important problem of finding
explicitly computable AO policies. This problem has not been treated in \cite{AtaBis}
beyond the linear case. Our goal here is to fill in this gap for the family
of costs described above.

In view of \eqref{10},
a sequence of control policies $B^n$ is said to be {\it asymptotically optimal}
if it achieves the limit, namely,
\begin{align}\label{ao}
\lim_{n\rightarrow \infty}J^n(B^n)=V\,.
\end{align}
The first main result that we prove is the asymptotic optimality of the preemptive
generalized $c\mu$ rule.
\begin{thm}\label{mainthm1}
Let Assumptions \ref{A1} and \ref{A2} hold. Then $B^{*,n}$ is asymptotically optimal.
\end{thm}
The nonpreemptive version of the generalized $c\mu$ rule is a control, denoted
$B^{\#,n}$, that upon completion of a job selects a customer from the class $i$ for which
$\tilde Q^n\in\calQ_i$. Namely, if $\tau$ is any time of departure (a jump time
of the process $\Del^n$), then $B^{\#,n}(\tau)=1_{\{\tilde Q^n(\tau)\in\calQ_i\}}$.
Note that the job departing at time $\tau$ is not counted in $\tilde Q^n(\tau)$,
due to right-continuity. Also note that if the system is empty right after a departure,
the above definition sets $B^{\#,n}(\tau)=0$, hence it is consistent with \eqref{52}.
One must also mention the non-idling condition:
when a customer is admitted into an empty system, it
is immediately served. Again, this set of conditions along with the equations
alluded to above uniquely define the processes involved; these details are skipped.

Our second main result requires slightly stronger assumptions.
\begin{thm}\label{mainthm2}
Let Assumption \ref{A3} hold. Then
$\displaystyle\lim_{n\rightarrow \infty}\hat{V}^n = V$. Furthermore,
$B^{\#,n}$ is asymptotically optimal.
\end{thm}

\begin{rem}
  \label{rem1}
{\bf (a)} The linear growth condition on $C$ is required to assure finiteness of the cost $J^n$.
The finiteness of $J^n$ is indeed guaranteed by means of Assumption \ref{A2}, by appealing to the linear growth condition on $C$ and the relation \eqref{systeq}. Furthermore, the linear growth condition is also used in step 5 in the proof of Theorem 4.2 in \cite{AtaBis}. The conditions on $C^{\prime}$ are necessary for the existence of a continuous minimizing curve (see Lemma 3.1) which plays the central role in the solution of the limiting differential game. Such conditions on $C^{\prime}$ have appeared in \cite{ManStol}, which has established the heavy-traffic AO of the generalized $\boldsymbol{c\mu}$ rule in the multi-server setup.

{\bf (b)} Functions $C$ that satisfy our assumptions include
\[
C(x)=\sum_i(a_i+b_ix_i^{p_i})^{1/{p_i}}-\sum_ia_i^{1/p_i},
\]
for constants $a_i>0$, $b_i>0$, $p_i>1$.
For example, one may take $C(x)=\sum_i(1+b_ix_i^2)^{1/2}-d$.
\end{rem}

\section{\bf Solution of the Game}\label{sec3}

In this section we give a minimizing strategy for the game described in the previous section. For that we require the following lemma about the existence of a continuous minimizing curve.
\begin{lem}\label{cont}
There exists a continuous function $f:\mathbb{R}_+\rightarrow \mathbb{R}^d_+$ such that
\begin{align}\label{minimizer}
\begin{cases} \theta\cdot f(w)=w,\\
 C(f(w))= \inf\{C(q):q \in \mathbb{R}^d_+,\theta\cdot q=w\},
\end{cases}
\qquad w\in\R_+.
\end{align}
This function satisfies
\begin{equation}\label{57}
\mu_1C^{\prime}_1(f_1(w))=\cdots=\mu_dC^{\prime}_d(f_d(w)),\qquad w\in\R_+.
\end{equation}
\end{lem}

\begin{proof}
First of all note that, by our assumptions, $C_i^{\prime}$ are continuous strictly increasing bounded functions, null at zero, for each $i$. Let $M_i= \mu_i\|C^{\prime}_i\|_{\infty}$. Without loss of generality let $M_1=\min M_i$. Then for any $c \in [0,M_1)$ there exists a $q^c=(q^c_1,\ldots,q^c_d)$ such that
$$
(\mu_1C^{\prime}_1(q^c_1),\ldots,\mu_dC^{\prime}_d(q^c_d))=c(1,\ldots,1)\,.
$$
Moreover, $q^0=(0,\ldots,0)$.
Now consider the function $F:[0,M_1)\rightarrow \mathbb{R}_+$ defined as
$$F(c)=\theta\cdot q^c\,.$$
Recall $\theta_i=\mu_i^{-1}$.
By our assumptions, $F$ is continuous, $F(0)=0$ and $\lim_{c\rightarrow M_1}F(c)=\infty$. So by the intermediate value theorem, for any $w \in \mathbb{R}_+$, there exists a $c$, such that $F(c)=w$.
Note that the strict monotonicity of $C'_i$ implies that of $c\mapsto q_i^c$, hence
that of $F$. As a result, the solution $c$ to $F(c)=w$ is unique.
Hence, for every $w \in \mathbb{R}_+$, there exists a unique member of $\R_+^d$,
that with an abuse of notation we denote by $q^w$, satisfying
\begin{equation}\label{54}
\mu_1C^{\prime}_1(q^w_1)=\cdots=\mu_dC^{\prime}_d(q^w_d)
\quad \mbox{and} \quad \theta \cdot q^w=w\,.
\end{equation}
Let $f:\mathbb{R}_+\rightarrow \mathbb{R}^d_+$ be defined
by $f(w)=q^w$. Note that, because of the facts that $C_i^{\prime}(0)=0$ and $C_i^{\prime}$ is strictly increasing, it follows from \eqref{54} that $f_i(w)>0$ whenever $w>0$. Thus by a Lagrange multiplier argument and by the strict convexity assumption we obtain that $f$ satisfies \eqref{minimizer}. Relation \eqref{57} holds by \eqref{54}.

Next we prove the continuity of $f$. Denote $\theta_{\min}=\min_i\theta_i$.
For that fix $\eps > 0$, and take $\delta = \eps\theta_{\min}/2$. We claim that
$|w-w^{\prime}|< \delta$ implies $\sum|q_i^w-q_i^{w^{\prime}}|< \eps$. Suppose not. Then there
exist $w$ and $w'$ with $0<w-w'<\del$ and $\sum|q^w_i-q^{w'}_i|\ge\eps$.
By the identity $w=\theta\cdot f(w)$, we have $w-w'=\sum\theta_i q^w_i - \sum \theta_iq_i^{w^{\prime}} < \delta$. But $\sum|\theta_i(q^w_i-q^{w'}_i)|\ge\eps\theta_{\min}=2\del$.
Note that if $a_i$ are any constants satisfying $0<\sum a_i<\del$ and $\sum|a_i|\ge2\del$
then there exist $i$ and $j$ for which $a_i<0<a_j$.
Therefore there must exist $i$ and $j$, such that $q_i^w> q_i^{w^{\prime}}$ and
$q_j^{w^{\prime}} > q_j^w$.
By the monotonicity of $C^{\prime}_i$ for each $i$, we have
$$C^{\prime}_i(q_i^w)>C^{\prime}_i(q_i^{w^{\prime}}) \quad \mbox{and} \quad
C^{\prime}_j(q_j^w)<C^{\prime}_j(q_j^{w^{\prime}})\,.$$ But this contradicts
\[
\mu_iC'_i(q_i^w)=\mu_jC'_j(q_j^w),\qquad \mu_iC'_i(q_i^{w'})=\mu_jC'_j(q_j^{w'}),
\]
which follows from \eqref{54}.
Hence we have proved the claim.
\end{proof}
Next we describe the solution of the game. For that consider the one-dimensional Skorohod map $\Gam$ from $\mathcal{D}([0,T],\mathbb{R})$ to itself given by
\begin{equation}\label{55}
\Gam[z](t)=z(t)-\inf_{s \in [0,t]}[z(s)\wedge0], \quad t \in [0,T]\,.
\end{equation}
From \eqref{dyn}, for $\psi=(\psi^{(1)},\psi^{(2)}) \in \mathcal{P}$ and $\zeta \in \mathcal{E}$, the dynamics of the differential game is given by $\varphi=\xi+\zeta$, where
$$\xi=y\io+\psi^{(1)}-\rho[\psi^{(2)}].
$$
We associate with each $\psi \in \mathcal{P}$ a $4$-tuple $(\hm{\varphi}[\psi],\hm{\xi}[\psi],\hm{\zeta}[\psi],\textbf{w}[\psi])$ given by
\begin{align*}
&\hm{\xi}[\psi]=y\io + \psi^{(1)}-\rho[\psi^{(2)}],&
\textbf{w}[\psi]=\Gam[\theta\cdot \hm{\xi}[\psi]],
\\
&\hm{\varphi}[\psi]=f(\textbf{w}[\psi]),&
\hm{\zeta}[\psi]=\hm{\varphi}[\psi]-\hm{\xi}[\psi].
\end{align*}
The following result has been proved in Proposition 3.1 of \cite{AtaBis}.
\begin{prop}
The map $\hm{\zeta}$ is an admissible strategy. Moreover, it is a minimizing strategy, namely,
$V=\sup_{\psi \in \mathcal{P}}c(\psi,\hm{\zeta}[\psi])$.
\end{prop}

\section{\bf Proof of Main Results}
\subsection{Proof of Theorem \ref{mainthm1}}\label{sec41}

In this section we keep Assumptions \ref{A1} and \ref{A2} in force and shorten the notation
$B^{*,n}$ to $B^n$. Our goal is to argue that
\begin{equation}\label{53}
\limsup_n J^n(B^n)\leq V.
\end{equation}
As we mentioned above, $\lim_nV^n=V$, and so Theorem \ref{mainthm1} will follow once \eqref{53}
is established.

To this end, given a constant $\Del$ define the set
$$\mathcal{D}_{\Del}=\{\psi \in\mathcal{D}([0,T],\mathbb{R}^{2d}): \mathbb{I}(\psi)\leq \Del\}\,.$$
Then by the definition of the rate function $\mathbb{I}$, $\mathcal{D}_{\Del}$ is a compact set containing absolutely continuous functions starting from zero, with derivative having $L^2$ norm uniformly bounded. By the compactness there exists a constant $M$ (depending on $\Del$) such that $\|\psi^{(1)}\|_T+\|\psi^{(2)}\|_T \leq M$ whenever $\psi\in\mathcal{D}_\Del$. Also, by the $L^2$ bound, the members of $\mathcal{D}_{\Del}$ are equicontinuous.
For any $\tilde{\psi} \in \mathcal{D}_{\Del}$ and any $r>0$ define
$$\mathcal{A}_r(\tilde{\psi})=\{\psi \in\mathcal{D}([0,T],\mathbb{R}^{2d}):d(\psi,\tilde{\psi})<r\}\,,$$ where $d$ is the Skorohod metric. We recall from \cite{Bill} that
$$
d(\psi, \tilde{\psi})=\inf_{g \in \Ups}\|g\|^\circ\vee
\|\psi-\tilde{\psi}\circ g\|_T,
$$
where $\Ups$ is the set of strictly increasing, continuous functions from $[0,T]$ onto itself,
and $$ \|g\|^\circ=\sup_{0\leq s< t \leq T}\biggl|\log\frac{g(t)-g(s)}{t-s}\biggr|\,. $$

Now, for any $g \in \Ups$,
\begin{align*}
\|\psi(t)-\tilde{\psi}(t)\|&\leq \|\psi(t)-\tilde{\psi}(g(t))\|+\|\tilde{\psi}(g(t))-\tilde{\psi}(t)\|,
\qquad t\in[0,T],\\
\|g(\cdot)-\cdot\|_T&\leq T\bigl(e^{\|g\|^\circ}-1\bigr).
\end{align*}
Thus by equicontinuity, for any $\eta >0$ there exists $r>0$ such that,
for any $\tilde{\psi} \in \mathcal{D}_{\Del}$,
\begin{align}\label{rel3}
\psi \in \mathcal{A}_r(\tilde{\psi})\quad \mbox{implies}\quad \|\psi -\tilde{\psi}\|_T < \eta\,.
\end{align}
Since $\mathcal{D}_\Del$ is compact
and $\mathbb{I}$ is lower semi-continuous, it is possible to choose a finite collection of members, $\psi^1,\psi^2,\ldots,\psi^N$, of $\mathcal{D}_{\Del}$, and positive constants $r_1,r_2,\ldots,r_N$,
such that, denoting  $\mathcal{A}^k= \mathcal{A}_{r_k}(\psi^k)$, one has
$\mathcal{D}_{\Del} \subset \cup_k\mathcal{A}^k$, and
\begin{align}
\inf\{\mathbb{I}(\psi): \psi \in \oo{\mathcal{A}^k}\}\geq \mathbb{I}(\psi^k)-\frac{\eps}{2},\quad \quad k=1,2,\ldots,N.
\end{align}
We fix such $\{\psi^k\}$ and $\{r_k\}$. By \eqref{rel3} we may, and will,
assume without loss of generality, that $\max_kr_k$ is so small that
\begin{equation}
  \label{59}
  \psi\in\mathcal{A}^k \quad \text{implies}\quad
  \sum_{j=1}^d\sum_{m=1}^2\|\psi^{(m)}_j-\psi^{k,(m)}_j\|_T\le\frac{\del}{8\theta_{\max}},
  \qquad k=1,2,\ldots,N,
\end{equation}
where $\theta_{\max}=\max_j\theta_j$, and $\del>0$ is chosen so that for $f$ as in Lemma \ref{cont},
\begin{equation}
  \label{61}
  \max_j\text{\osc}_\del(f_j,L)<\frac{\eps}{4},
  \qquad L:=4\theta_{\max}\Big(\|y\|_1T+d(M+1)\Big).
\end{equation}

For $k=1,2,\ldots,N$, let
$$(\varphi^k,\xi^k,\zeta^k,w^k)=(\hm{\varphi}[\psi^k],\hm{\xi}[\psi^k],\hm{\zeta}[\psi^k],\textbf{w}[\psi^k])\,.$$
Thus $\varphi^k$ is the dynamics corresponding to $\psi^k$ and $\zeta^k$. Let $\La^n=\|\tilde{A}^n\|_T+\|\tilde{S}^n\|_T$, and define
\begin{align}
\Om^n_k=\{(\tilde{A}^n,\tilde{S}^n)\in \mathcal{A}^k\},\quad\quad k=1,2,\ldots,N\,.
\end{align}
Fix $\eps> 0$. Then as in the proof of Theorem 4.2 in \cite{AtaBis},
it is enough to show that there exists a constant $c_1$ such that for all $n$ sufficiently large,
\begin{align}\label{est1}
\|\tilde{Q}^n\|_T\leq c_1(1+\La^n)
\end{align}
and
\begin{align}\label{est2}
\|\tilde{Q}^n-\varphi^k\|_T \leq c_1\eps \quad \quad \mbox{on}\,\,\Om^n_k, k=1,2,\ldots,N \,.
\end{align}Once the two estimates are established, the rest of the proof proceeds as in
Step 5 of Theorem 4.2 in \cite{AtaBis}.
To this end,
let $\theta^n=((\mu^n_1)^{-1}n,\ldots,(\mu^n_d)^{-1}n)$. Then, it follows from \eqref{systeq}
and \eqref{control} that
\begin{align}\label{rel2}
\theta^n\cdot \tilde{Q}^n(t)=\theta^n\cdot Y^n(t) + \frac{\sqrt{n}}{b_n}
\int_0^t1_{\{\theta^n\cdot \tilde{Q}^n(s)=0\}}ds\,,
\end{align}
where $Y^n_i(t)=y^n_i t + \tilde{A}^n_i(t) - \tilde{S}^n_i(T^n_i(t))$.
Since $\theta^n\cdot \tilde{Q}^n$ is non-negative and the last term on the right hand side increases
only on the set of times $\{t:\theta^n\cdot \tilde{Q}^n(t)=0\}$, it follows
that $\theta^n\cdot\tilde Q^n$ is a solution to the
one-dimensional Skorohod problem for $\theta^n\cdot Y^n$. In particular,
\[
\theta^n\cdot\tilde Q^n=\Gam[\theta^n\cdot Y^n].
\]
Thus by \eqref{55},
\begin{align}\label{est3}
\|\theta^n\cdot \tilde{Q}^n\|_T+\| \theta^n\cdot Z^n\|_T \leq 3\|\theta^n\cdot Y^n\|_T\,.
\end{align}
By the assumptions \eqref{56} on the limit parameters, and the definitions of the constants
$\theta^n$, $\theta$, $y^n$ and $y$, one has $\theta^n\to\theta$ and $y^n\to y$.
Using the nonnegativity of $\tilde{Q}^n$,
it follows from \eqref{est3} that \eqref{est1} holds,
for a suitable constant $c_1$ that does not depend on $n$.

Note that $\La^n$ is bounded by $M+1$ on $\cup_k\Om^n_k$.
As a result, we also have the following uniform bound on $\tilde Q^n$,
namely, for all large $n$, $\|\tilde Q^n\|_T\le c_2$ on the event $\cup_k\Om^n_k$,
where $c_2$ does not depend on $n$ or $\eps$ (but may depend on $\Del$).
A more concrete bound will be needed for $\theta^n\cdot\tilde Q^n$.
Namely, on $\cup_k\Om^n_k$,
\begin{align}\label{60}\notag
\|\theta^n\cdot\tilde Q^n\|_T &\le3\|\theta^n\cdot Y^n\|_T
\le4\theta_{\max}\Big(\|y\|_1T+\sum_{j=1}^d(\|\tilde A^n_j\|_T+\|\tilde S^n_j\|_T)\Big)\\
&\le4\theta_{\max}\Big(\|y\|_1T+d(M+1)\Big)=L,
\end{align}
by \eqref{61}.

We next prove \eqref{est2}. Let
\[
G^n_i(t)=\frac{\theta^n_i}{\theta_i}\tilde{Q}^n_i(t)-f_i(\theta^n\cdot \tilde{Q}^n(t)),
\qquad t\in[0,T],\,i\in\calI.
\]
We first show that, for all $n$ large, on $\Om^n_k$, one has
\begin{align}\label{est9}
\|G^n_i\|_T < c_3\eps, \quad \quad \text{for all } i,
\end{align}
for some constant $c_3$ that does not depend on $n$ or $\eps$.
Note that a sufficient condition for \eqref{est9} is that, for each $i$,
\begin{align}\label{est4}
\inf_{0\leq t\leq T}G^n_i(t)\geq -\eps\,.
\end{align}
To see this, recall that from the definition of $f$ we have
\begin{align}\label{rel1}
\sum_{i}\theta_iG^n_i=0.
\end{align}
Hence, if \eqref{est4} holds then
\begin{align*}
\theta_iG^n_i=-\sum_{j:j\neq i}\theta_jG^n_j\leq \eps\sum_{j:j\neq i}\theta_j\,,
\end{align*}
by which \eqref{est9} holds.

We thus turn to prove \eqref{est4}. Fix $i$.
Arguing by contradiction, assume that \eqref{est4} is false.
Recall that we assume the initial condition $\tilde Q^n(0)=0$.
Moreover, the jump sizes of the process $\tilde Q^n$, hence those of $G^n$, are
uniformly small when $n$ is large, by appealing to the continuity of $f$ and the
uniform bound on $\tilde Q^n$ alluded to above.
By these considerations, it follows that, provided $n$ is sufficiently large,
there must exist times $\sigma^n$ and $\tau^n$ such that $0\le\sigma^n<\tau^n\le T$, and
\begin{align}\label{est5}
G^n_i(\tau^n)< -\eps\,,
\qquad
G^n_i(\sigma^n)\geq -\frac{\eps}{2}\,,
\qquad
G^n_i(t) < -\frac{\eps}{4} \text{ for all } t\in[\sigma^n,\tau^n].
\end{align}
It follows from the last assertion of \eqref{est5} that,
for all $n$ sufficiently large, for $t\in[\sigma^n,\tau^n]$,
\begin{align}\label{est6}
\tilde{Q}^n_i(t)-f_i(\theta^n\cdot \tilde{Q}^n(t))< 0\,,
\end{align}
where we used the uniform bound on $\tilde Q^n$ and the convergence $\theta^n\to\theta$.
Also, by \eqref{rel1} and \eqref{est5}, for every $t \in [\sigma^n,\tau^n]$
there exists a $j$ (depending on $t$) such that $G^n_j(t)> c^{\prime}\eps$
for some constant $c^{\prime}>0$. Therefore, provided that $n$ is large,
for every $t\in[\sig_n,\tau_n]$ there exists a $j$ such that
\begin{align}\label{est7}
\tilde{Q}^n_j(t)-f_j(\theta^n\cdot \tilde{Q}^n(t))> 0\,.
\end{align}
By \eqref{57},
$$
\mu_iC^{\prime}_i(f_i(\theta^n\cdot \tilde{Q}^n))=\mu_jC^{\prime}_j(f_j(\theta^n\cdot
\tilde{Q}^n))\,.
$$
Using \eqref{est6} and \eqref{est7} and the strict
monotonicity of $C^{\prime}_m$ for all $m$,
we have that
\begin{align}\label{62}
\text{for every $t \in [\sigma^n,\tau^n]$ there exists a $j$ such that }
\mu_iC^{\prime}_i(\tilde{Q}^n_i(t))<\mu_jC^{\prime}_j(\tilde{Q}^n_j(t))\,.
\end{align}
By the way the policy is defined, specifically \eqref{control} and \eqref{58}, we obtain
that on the interval $[\sigma^n,\tau^n]$ class $i$ does not receive any service,
namely $B^n_i=0$.

The first two statements of \eqref{est5} yield
\begin{align}\label{est8}
\frac{\theta^n_i}{\theta_i}(\tilde{Q}^n_i(\tau^n)-\tilde{Q}^n_i(\sigma^n))
+\frac{\eps}{2}\leq f_i(\theta^n\cdot \tilde{Q}^n(\tau^n))-f_i(\theta^n\cdot \tilde{Q}^n(\sigma^n))\,.
\end{align}
Fix a sequence $\{a_n\}$ having the properties that $a_n\rightarrow 0$
but $\sqrt{n}b_n^{-1}a_n \rightarrow \infty$. We distinguish between two cases.

\noi
\textit{Case 1:} $\tau^n-\sigma^n\leq a_n$.\\
Since on the interval $[\sigma^n,\tau^n]$ class $i$ receives no service,
the queue length can only increase over the interval, as follows from the balance equation
\eqref{ioeq}, \eqref{50}, and $B^n_i=0$. Thus
$\tilde{Q}^n_i(\tau^n)\ge\tilde{Q}^n_i(\sigma^n)$.
We obtain from \eqref{est8} that
\begin{align}\label{63}
f_i(\theta^n\cdot \tilde{Q}^n(\tau^n))-f_i(\theta^n\cdot \tilde{Q}^n(\sigma^n))\geq \frac{\eps}{2}>\frac{\eps}{4}.
\end{align}
(Working with the bound $\eps/4$ instead of $\eps/2$
will be useful in the proof provided in Section \ref{sec42}.)
Recalling that $\theta^n\cdot\tilde Q^n$ is uniformly bounded by $L$ (see \eqref{60}),
we have from \eqref{61} that
\begin{align*}
\theta^n\cdot \tilde{Q}^n(\tau^n)-\theta^n\cdot \tilde{Q}^n(\sigma^n)\geq \del\,,
\end{align*}
where $\del$ is as in \eqref{61}.
From \eqref{est6}, it follows that in the interval $[\sigma^n,\tau^n]$,
$\theta^n\cdot \tilde{Q}^n >0$, because the function $f$ vanishes at zero.
Thus by \eqref{rel2},
\begin{align*}
\theta^n\cdot(Y^n(\tau^n)-Y^n(\sigma^n))\geq \del\,.
\end{align*}
Hence, using the definition of $Y^n$, the convergence $\theta^n\to\theta$,
and the fact that $T^n_i$ is Lipschitz with constant $1$, we obtain
\begin{align*}
\del\le2\theta_{\max}\Big[2\|y\|_1a_n+\sum_{j=1}^d \text{osc}_{a_n}(\tilde{A}^n_j)
+\sum_{j=1}^d \text{osc}_{a_n}(\tilde{S}^n_j)\Big]\,.
\end{align*}
Now, on $\Om^n_k$,
$$
\mbox{osc}_{a_n}(\tilde{A}^n_j)\leq 2\|\tilde{A}^n_j-\psi^{k,(1)}_j\|_T
+\mbox{osc}_{a_n}(\psi^{k,(1)}_j).
$$
A similar statement holds for $\tilde{S}^n$.
Hence, on $\Om^n_k$,
\[
\del\le 2\theta_{\max}\Big[2\|y\|_1a_n+\frac{\del}{4\theta_{\max}} +\sum_{j=1}^d\sum_{m=1}^2
\{\text{osc}_{a_n}(\psi^{k,(m)}_j)\}\Big],
\]
where we used \eqref{59}.
By the equicontinuity of the functions in the class $\mathcal{D}_{\Del}$,
we obtain $\del\le\del/2+\eta_n$ for some $\eta_n\to0$. Hence,
for all $n$ large enough we arrive at a contradiction.

\noi
\textit{Case 2:} $\tau^n-\sigma^n\geq a_n$.\\
On $\Om^n_k$, the right hand side of \eqref{est8} is bounded.
As for its left hand side, we have from \eqref{systeq} that
\begin{align}\label{est12}
\nonumber \tilde{Q}^n_i(\tau^n)-\tilde{Q}^n_i(\sigma^n)&=y^n_i(\tau^n-\sigma^n)
+ \tilde{A}^n_i(\tau^n)-\tilde{A}^n_i(\sigma^n)
-(\tilde{S}^n_i(T^n_i(\tau^n))-\tilde{S}^n_i(T^n_i(\sigma^n)))
\\
&\quad +\frac{\mu^n_i}{n}\frac{\sqrt{n}}{b_n}\rho_i(\tau^n-\sigma^n)\,,
\end{align}
where we used the fact that $T^n_i$ does not increase on the interval $[\sig_n,\tau_n]$.
On the right hand side of \eqref{est12},
all terms are bounded, except the last term
which tends to infinity as $n\to\iy$.
Thus again for large $n$ we arrive at a contradiction.

Summarizing the two cases, we conclude that \eqref{est4},
hence \eqref{est9}, holds on $\Om^n_k$ provided that $n$ is sufficiently large.

Finally, we use \eqref{est9} to show \eqref{est2}.
Since $\theta^n \rightarrow \theta$ and $\tilde{Q}^n$ is bounded on $\Om^n_k$,
there exists a constant $c_4$ such that on $\Om^n_k$, for all large $n$,
\begin{align}\label{est10}
\|\tilde{Q}^n-f(\theta^n\cdot\tilde{Q}^n)\|_T< c_4\eps\,.
\end{align}
Now using the facts that $\theta^n\cdot \tilde{Q}^n=\Gam[\theta^n\cdot Y^n]$
and $\theta\cdot\varphi^k=\Gam[\theta\cdot\xi^k]$, where
$\xi^k=y\io+\psi^{k,(1)}-\rho[\psi^{k,(2)}]$,
and the Lipschitz continuity of the map $\Gam$ we have for all $n$ sufficiently
large,
\begin{align*}
\|\theta^n\cdot \tilde{Q}^n-\theta\cdot \varphi^k\|_T &\leq 2\|\theta^n\cdot Y^n-\theta^n\cdot\xi^k\|_T+2\|\theta^n-\theta\|_T\\
&\leq c_5\|Y^n-\xi^k\|_T+\eps\\
&\leq c_5[\|\tilde{A}^n-\psi^{k,(1)}\|_T+\|\tilde{S}^n\circ T^n -\rho[\psi^{k,(2)}]\|_T]+2\eps\,.
\end{align*}Now by \eqref{est3} and the definition of $Z^n$ we have
$$
\|\rho_i \io-T^n_i\|_T\rightarrow 0,\quad\quad \mbox{as}\quad n\rightarrow \infty\,.
$$
Thus on $\Om^n_k$, for all sufficiently large $n$, there exists a constant $c_6$ such that,
\begin{align}
\|\theta^n\cdot\tilde{Q}^n-\theta\cdot\varphi^k\|_T<c_6\eps\,.
\end{align}
Now, from the definition of $\varphi^k$ we have $\varphi^k=f(\theta\cdot\varphi^k)$.
Since on $\Om^n_k$, $\tilde{Q}^n$ and $\varphi^k$ are bounded,
by the continuity of $f$ there exists a constant $c_7$ such that
\begin{align}\label{est11}
\|f(\theta^n\cdot\tilde{Q}^n)-\varphi^k\|_T < c_7\eps.
\end{align}
Combining \eqref{est10} and \eqref{est11}, we obtain \eqref{est2}. This concludes the proof.
\qed

\subsection{Proof of Theorem \ref{mainthm2}}\label{sec42}

Since the control space $\mathcal{B}^{\#,n}$ is contained in $\mathcal{B}^n$, we automatically have
$\hat{V}^n \geq V^n$. Hence it suffices to prove the asymptotic optimality of
the non-preemptive generalized $c\mu$ rule, $B^{\#,n}$.
The proof for that follows the same steps as for the preemptive case with some modifications,
which we describe next. The first deviation from the proof of Theorem \ref{mainthm1}
is related to \eqref{62}. While \eqref{62} is a valid statement, we cannot deduce
from it that $B^n_i=0$ on the interval $[\sig^n,\tau^n]$ (here and in what follows,
we suppress the symbol $\#$ in $B^{\#,n}$).
Indeed, \eqref{62} assures that every time the server {\it becomes available}
during this interval customers from classes other than $i$ are selected for service, but
it is possible that at the time $\sig^n$ a class-$i$ customer is already in service. In this case,
one has $B^n_i=1$ until this service completes.
Note that \eqref{est8} remains valid.
In what follows, we describe how Cases 1 and 2 of the proof of Theorem \ref{mainthm1}
are modified to address this change.
With the modified version of these two cases, the remainder of the proof proceeds as before.

For Case 1, one can no longer say that class $i$ receives no service. However,
at most one customer of that class is served during $[\sig^n,\tau^n]$, by which
$\tilde Q^n_i(\tau^n)-\tilde Q^n_i(\sig^n)\ge -b_n^{-1}n^{-1/2}$. With \eqref{est8},
using $\theta^n_i\to\theta_i$ and assuming $n$ is large, we obtain
\begin{align*}
f_i(\theta^n\cdot \tilde{Q}^n(\tau^n))-f_i(\theta^n\cdot \tilde{Q}^n(\sigma^n))\geq
\frac{\eps}{2}-\frac{2}{b_n\sqrt{n}}>\frac{\eps}{4}.
\end{align*}
As a result, the final conclusion of \eqref{63} is still valid, and we can proceed
exactly as before.

As for Case 2, let
\[
\ST^n_j=\max_{l\le A^n_j(T)}\frac{\ST_j(l)}{\mu^n_j}
\]
denote the maximal class-$j$ service time among all jobs arriving up to time $T$.
Equation \eqref{est12} need not hold because of the possible customer in service at time
$\sig^n$. However, the time devoted to this customer
is bounded by $\ST^n_i$. As a result, in place of \eqref{est12}, we may write
\begin{align*}
\tilde{Q}^n_i(\tau^n)-\tilde{Q}^n_i(\sigma^n)&\ge
y^n_i(\tau^n-\sigma^n)
+ \tilde{A}^n_i(\tau^n)-\tilde{A}^n_i(\sigma^n)
-(\tilde{S}^n_i(T^n_i(\tau^n))-\tilde{S}^n_i(T^n_i(\sigma^n)))
\\
&\quad +\frac{\mu^n_i}{n}\frac{\sqrt{n}}{b_n}\rho_i(\tau^n-\sigma^n)
-\frac{\mu^n_i}{n}\frac{\sqrt{n}}{b_n}\rho_i\ST^n_i.
\end{align*}
The proof will proceed as before if we show that the last term in the above display
is bounded off of an event of negligible probability. More precisely, it suffices to show
that for every constant $c_1$, we have for large $n$,
\begin{equation}
  \label{64}
  \PP\Big(\frac{\sqrt{n}}{b_n}\max_j\ST^n_j> 1\Big)\le e^{-c_1b_n^2}.
\end{equation}

To this end, note first that,
since $\tilde{A}^n$ satisfies the MDP, there exists a constant $c_3$ such that for any constant
$c_1$, for all $n$ large enough,
$$
\mathbb{P}(\max_jA^n_j(T)> c_3 n)\leq e^{-c_1b_n^2}\,.
$$
Then with $c_3$ as above we have,
\begin{align*}
\mathbb{P}\Big(\ST^n_j> \frac{b_n}{\sqrt{n}}\Big)
&\leq c_3 n \mathbb{P}\Big(\ST_j(1) > \frac{\mu_j^n b_n}{\sqrt{n}}\Big)
+\mathbb{P}(A^n_j(T))>c_3 n)\\
&\leq c_3 n \mathbb{P}\Big(\ST_j(1) > \frac{\mu_j b_n\sqrt{n}}{2}\Big)
+e^{-c_1b_n^2}\\
&\leq c_4 n e^{-\frac{u_0\mu_j}{2}(b_n \sqrt{n})}+e^{-c_1b_n^2}\\
&\le 2e^{-c_1b_n^2}\,,
\end{align*}
where $c_4$ is a constant, and the second last line follows from Assumption \ref{A3}.
It follows that for any $c_1$, \eqref{64} holds for all n large enough.\\
This completes the proof.
\qed


\skp

\end{document}